\documentclass[11pt]{amsart}
\pdfoutput=1
\usepackage{mathpazo}
\usepackage{latexsym}   
\usepackage{amssymb}    
\usepackage{amsmath}    
\usepackage{amsthm}     
\usepackage{booktabs}
\usepackage[figuresright]{rotating}
\usepackage[colorlinks, breaklinks=true]{hyperref}
\usepackage[margin=1in]{geometry}
\usepackage[all]{xy}

\newtheorem{theorem}{Theorem}[section]

\DeclareMathOperator{\degree}{deg}
\DeclareMathOperator{\different}{{\mathfrak d}}
\DeclareMathOperator{\divisor}{div}

\newcommand{\BF}{{\mathbb{F}}}

\renewcommand{\le}{\leqslant}
\renewcommand{\ge}{\geqslant}

\newcommand{\p}{\phantom{0}} 

\newcommand{\Drinfeld}{Drinfel'd}
\newcommand{\Vladut}{Vl\u adu\c t}

\makeatletter
\@namedef{subjclassname@2020}{%
  \textup{2020} Mathematics Subject Classification}
\makeatother

\begin{document}

\title[The maximum number of points on a curve of genus eight]
{The maximum number of points on a curve of genus eight\\ over the field of four elements}

\author{Everett W.~Howe}
\address{Unaffiliated mathematician, San Diego, CA, USA}
\urladdr{\url{http://ewhowe.com}}
\email{\href{mailto:"Everett W. Howe" <however@alumni.caltech.edu>}{however@alumni.caltech.edu}}

\date{14 August 2020}


\begin{abstract}
The Oesterl\'e bound shows that a curve of genus $8$ over the finite field $\BF_4$ 
can have at most $24$ rational points, and Niederreiter and Xing used class
field theory to show that there exists such a curve with $21$ points. We improve
both of these results: We show that a genus-$8$ curve over $\BF_4$ can have at 
most $23$ rational points, and we provide an example of such a curve with $22$ 
points, namely the curve defined by the two equations 
$y^2 + (x^3 + x + 1)y = x^6 + x^5 + x^4 + x^2$ and $z^3 = (x+1)y + x^2.$
\end{abstract}

\maketitle

\section{Introduction}

For the past several decades --- beginning with the work of 
Ihara~\cite{Ihara1981}, Manin~\cite{Manin1981}, and \Drinfeld\ and 
\Vladut~\cite{VladutDrinfeld1983r} in the early 1980s --- there has been much
research on the behavior of the quantity $N_q(g)$, the maximum number of 
rational points on a curve\footnote{
  Here, and throughout the paper, by a \emph{curve} over a field $k$
  we mean a smooth, projective, geometrically irreducible variety over $k$
  of dimension~$1$.}
of genus $g$ over the finite field~$\BF_q$. Some of
this research is focused on the asymptotic behavior of $N_q(g)$, where one of
$q$ and $g$ is fixed and the other tends to infinity, while other research is
more concerned with the computation of the actual values of $N_q(g)$ for small 
$q$ and~$g$. These actual values provide a source of data that can inspire and 
test conjectures about the behavior of $N_q(g)$, and curves whose point counts 
attain these values can be used to create efficient error-correcting codes. 
In short, determining the value of $N_q(g)$ for specific $q$ and $g$ is an
attractive and challenging mathematical problem.

The web site \href{http://manypoints.org}{manypoint.org} \cite{manypoints} keeps
track of the known upper and lower bounds for $N_q(g)$, for $g\le 50$ and for
$q$ ranging over the primes less than $100$, the prime powers $p^i$ for odd 
$p\le 19$ and $i\le 5$, and the powers of $2$ up to $2^7$. The extent of the
research activity centered around the study of the function $N_q(g)$ is 
illustrated by the fact that as of this writing, there are $115$ references
cited on manypoints.org, involving $81$ authors. For most values of $q$ and $g$ 
in the given ranges, the exact value of $N_q(g)$ is not known, but for small $q$
there has been more focused study and searching. The exact value of $N_2(g)$ is 
known for all $g\le 11$, and $N_2(12)$ is equal to either $14$ or $15$; the
exact value of $N_3(g)$ is known for $g\le 7$, and $N_3(8)$ is either $17$ 
or~$18$; and the exact value of $N_4(g)$ is known for $g\le 7$, while for $g=8$
all we know is that $21\le N_4(8)\le 24$. The upper bound $N_4(8)\le 24$ comes 
from Oesterl\'e's ``explicit formula'' method, described 
in~\cite[p. SeTh 32 ff.]{Serre:notes} and~\cite[Chapter VI]{Serre2020}. The 
lower bound $N_4(8)\ge 21$ comes from a construction of Niederreiter and 
Xing~\cite{NiederreiterXing1997}, who use class field theory to produce a
genus-$8$ curve over $\BF_4$ with $21$ points that is an unramified degree-$7$
cover of a genus-$2$ curve.

The purpose of this note is to improve the upper and lower bounds on $N_4(8)$. 

\begin{theorem}
\label{T:intro}
We have $22\le N_4(8)\le 23$. 
\end{theorem}

To prove the lower bound, we exhibit a genus-$8$ curve over $\BF_2$ that has
$22$ points over $\BF_4$; the curve is constructed as a degree-$3$ Kummer 
extension of a genus-$2$ curve. For the upper bound, we show that there are $26$
isogeny classes of $8$-dimensional abelian varieties over $\BF_4$ that might
possibly contain the Jacobian of a genus-$8$ curve having $24$ rational points.
Using techniques and programs from~\cite{HoweLauter2012} we show that $25$ of 
these isogeny classes do not contain Jacobians. For the final isogeny class, the
techniques of~\cite{HoweLauter2012} show that any curve whose Jacobian lies in
the isogeny class must be a triple cover of the unique elliptic curve $E$ over
$\BF_4$ with $8$ points. Using methods related to ones used by 
Rigato~\cite[\S4]{Rigato2010}, we are able to show that no such triple cover can
exist.

Our result leaves open the question of the exact value of $N_4(8)$. We have been
unable to construct a genus-$8$ curve over $\BF_4$ having $23$ rational points.
On the other hand, if we try to use the techniques of~\cite{HoweLauter2012} to 
show that there are no genus-$8$ curves with $23$ points, we are left to 
consider $98$ isogeny classes of abelian $8$-folds that are not eliminated by
other means, a daunting prospect. Finding the exact value of $N_4(8)$ will 
likely require significant new techniques.

\section*{Acknowledgments}
The author thanks Jeroen Sijsling for his careful reading of the first version
of this paper, and for his suggestions for improving the exposition.

\section{Improving the upper bound}
\label{S:HL}

In this section we prove the following result:

\begin{theorem}
\label{T:upper}
There is no genus-$8$ curve over $\BF_4$ that has $24$ rational points.
\end{theorem}

Combined with the Oesterl\'e bound $N_4(8)\le 24$, this theorem gives us the 
upper bound in Theorem~\ref{T:intro}.

\begin{proof}[Proof of Theorem~\textup{\ref{T:upper}}]
To obtain a contradiction, suppose such a curve $C$ exists. We will determine
the real Weil polynomial (defined below) of $C$, and use this polynomial to show
that $C$ must be a triple cover of a certain elliptic curve. Then we will show 
that no such triple cover can exist.

The \emph{Weil polynomial} of an abelian variety over a finite field is the 
characteristic polynomial of its Frobenius endomorphism, and a result of 
Tate~\cite[Theorem 1(c), p.~139]{Tate1966} says that two abelian varieties over
the same finite field $k$ are isogenous to one another (over $k$) if and only if
they have the same Weil polynomial. If $A$ is a $g$-dimensional abelian variety
over $\BF_q$ with Weil polynomial $f$, then the \emph{real Weil polynomial} of 
$A$ is the unique degree-$g$ polynomial $h$ such that $f(x) = x^g h(x + q/x)$. 
We define the Weil polynomial and the real Weil polynomial of a curve over a 
finite field to be the Weil polynomial and the real Weil polynomial of the 
curve's Jacobian, respectively.

Suppose $f$ is the Weil polynomial of an isogeny class of $g$-dimensional 
abelian varieties over $\BF_q$, and let $\pi_1,\ldots,\pi_{2g}$ be the complex
roots of~$f$, listed with multiplicity. For every $n>0$ we define quantities 
$R_n(f)$ and $P_n(f)$ by
\[
R_n(f) = q^n + 1 - \sum_{i=1}^{2g} \pi_i^n \text{\qquad and\qquad}
P_n(f) = \frac{1}{n} \sum_{d\mid n} R_d(f) \, \mu\Bigl(\frac{n}{d}\Bigr),
\]
where $\mu$ is the M\"obius function. If $f$ is the Weil polynomial of a 
curve~$C$, then $R_n(f) = \#C(\BF_{q^n})$ and $P_n(f)$ is equal to the number of
degree-$n$ places on $C$, and so $P_n(f)\ge 0$. This gives us a simple test that
must be satisfied if an isogeny class of abelian varieties contains a Jacobian: 
We simply check to see that $P_n(f)\ge 0$ for all $n$. An isogeny class that 
fails this test does not contain a Jacobian, while an isogeny class that passes
the test may or may not contain a Jacobian.

Lauter~\cite{Lauter2000} sketches an algorithm for enumerating the Weil 
polynomials $f$ for the isogeny classes of $g$-dimensional abelian varieties
over $\BF_q$ with $P_n(f)\ge 0$ for all $n$ and with $P_1(f)$ equal to a given
integer~$N$. The Jacobian of a genus-$g$ curve over $\BF_q$ with $N$ points will
necessarily lie in one of these isogeny classes. The paper~\cite{HoweLauter2012}
and its predecessor~\cite{HoweLauter2003} provide many methods for showing that 
certain isogeny classes that pass this initial test do not contain Jacobians; we
mention three here that are stated in terms of real Weil polynomials. One 
criterion, due to Serre~\cite[p.~Se 11 ff.]{Serre:notes}
\cite[\S II.4]{Serre2020}, says that if the real Weil polynomial $h$ of an 
isogeny class can be written as a product of two nontrivial monic factors $h_1$
and $h_2$ such that the resultant of $h_1$ and $h_2$ is $\pm 1$, then the 
isogeny class does not contain a Jacobian (see also
\cite[Theorem~2.2(a) and Proposition~2.8]{HoweLauter2012}). A second 
criterion~\cite[Theorem~2.2(b) and Proposition~2.8]{HoweLauter2012} says that if
we have a factorization $h = h_1 h_2$ where the resultant of $h_1$ and $h_2$ is
equal to $\pm2$ (or more generally, when the \emph{gluing exponent} of the pair
of isogeny classes associated to $h_1$ and $h_2$ is equal to~$2$), then any 
curve with real Weil polynomial $h$ is a double cover of a curve whose real Weil 
polynomial is either $h_1$ or $h_2$; often, this added information is enough to
show that there can be no curve with real Weil polynomial~$h$. And finally, if
$q$ is a square, say $q = s^2$, and if the real Weil polynomial $h$ of an 
isogeny class can be factored as $h = h_1 h_2$ where $h_1$ is a power of 
$x - 2s$ and where $h_2(2s)$ is squarefree, then there is no Jacobian with real
Weil polynomial~$h$~\cite[Theorem~3.1]{HoweLauter2012}. We will call these three
techniques the \emph{resultant~$1$} argument, the \emph{resultant~$2$} argument,
and the \emph{supersingular factor} argument, respectively.

Using the Magma programs that accompany~\cite{HoweLauter2012}, which may be
found at 

\smallskip

\centerline{\href{http://ewhowe.com/papers/paper35.html}{\texttt{http://ewhowe.com/papers/paper35.html}},}

\smallskip
\noindent
we can analyze the possible isogeny classes of abelian $8$-folds over $\BF_4$ 
that meet the ``non-negative place count'' test and that might contain the 
Jacobian of a genus-$8$ curve over $\BF_4$ with $24$ rational points. There are
$26$ such isogeny classes, which we enumerate in Table~\ref{Table:26}.

\begin{sidewaystable}
\caption{
   The real Weil polynomials of isogeny classes of abelian $8$-folds over 
   $\BF_4$ that might contain the Jacobian of a genus-$8$ curve having $24$
   points. All but the first of the $26$ polynomials can be eliminated from
   consideration by arguments discussed in the text: the ``resultant $1$'' 
   argument, the ``resultant $2$'' argument, and the ``supersingular factor''
   argument.
}
\label{Table:26}
\centering
\renewcommand{\arraystretch}{1.2}
\scriptsize
\begin{tabular}{rl@{\qquad}lr@{\ $\times$\ }l@{}}
\toprule
No. & Real Weil polynomial $h$                                                          & Argument      & Splitting $h_1$                & $h_2$\\
\midrule
 1. & $x^8 + 19x^7 + 152x^6 + 664x^5 + 1712x^4 + 2608x^3 + 2176x^2 + {\p}768x$          & \qquad --- \\
 2. & $x^8 + 19x^7 + 152x^6 + 664x^5 + 1713x^4 + 2618x^3 + 2212x^2 + {\p}824x + {\p}32$ & Resultant $1$ & $(x + 2)^3 (x + 4)$            & $(x^4 + 9x^3 + 26x^2 + 24x + 1)$                         \\
 3. & $x^8 + 19x^7 + 152x^6 + 664x^5 + 1713x^4 + 2619x^3 + 2220x^2 + {\p}844x + {\p}48$ & Resultant $1$ & $(x + 2)^2 (x + 4)$            & $(x + 3) (x^4 + 8x^3 + 20x^2 + 16x + 1)$                 \\
 4. & $x^8 + 19x^7 + 152x^6 + 664x^5 + 1714x^4 + 2629x^3 + 2256x^2 + {\p}900x + {\p}80$ & Resultant $1$ & $(x^2 + 5x + 5)$               & $(x + 2)^2 (x + 4) (x^3 + 6x^2 + 9x + 1)$                \\
 5. & $x^8 + 19x^7 + 152x^6 + 664x^5 + 1715x^4 + 2640x^3 + 2299x^2 + {\p}970x +    120$ & Resultant $1$ & $(x + 3)(x^2 + 5x + 5)$        & $(x + 2)(x + 4)(x^3 + 5x^2 + 6x + 1)$                    \\
 6. & $x^8 + 19x^7 + 152x^6 + 664x^5 + 1715x^4 + 2641x^3 + 2308x^2 + {\p}996x +    144$ & Resultant $2$ & $(x + 2)^2 (x^2 + 4x + 1)$     & $(x + 1)(x + 3)^2(x + 4)$                                \\
 7. & $x^8 + 19x^7 + 152x^6 + 664x^5 + 1715x^4 + 2642x^3 + 2317x^2 +    1022x +    168$ & Resultant $1$ & $(x + 2)(x + 4)(x^2 + 3x + 1)$ & $(x + 3) (x^3 + 7x^2 + 14x + 7)$                         \\
 8. & $x^8 + 19x^7 + 152x^6 + 664x^5 + 1716x^4 + 2650x^3 + 2335x^2 +    1025x +    150$ & Resultant $1$ & $(x^2 + 5x + 5)^2$             & $(x + 2)(x + 3)(x^2 + 4x + 1)$                           \\
 9. & $x^8 + 19x^7 + 152x^6 + 664x^5 + 1716x^4 + 2651x^3 + 2343x^2 +    1045x +    165$ & Resultant $1$ & $(x + 3)$                      & $(x^2 + 5x + 5)(x^5 + 11x^4 + 44x^3 + 77x^2 + 55x + 11)$ \\
10. & $x^8 + 19x^7 + 152x^6 + 664x^5 + 1716x^4 + 2652x^3 + 2351x^2 +    1065x +    180$ & S.s. factor   & $(x+4)$                        & $(x + 1)(x + 3)^2(x^2 + 3x + 1)(x^2 + 5x + 5)$           \\
11. & $x^8 + 19x^7 + 152x^6 + 664x^5 + 1716x^4 + 2652x^3 + 2352x^2 +    1071x +    189$ & Resultant $1$ & $(x^3 + 7x^2 + 14x + 7)$       & $(x + 3)^2(x^3 + 6x^2 + 9x + 3)$                         \\
12. & $x^8 + 19x^7 + 152x^6 + 664x^5 + 1716x^4 + 2652x^3 + 2352x^2 +    1072x +    192$ & Resultant $1$ & $(x + 3)$                      & $(x + 2)^2 (x + 4) (x^2 + 4x + 2)^2$                     \\
13. & $x^8 + 19x^7 + 152x^6 + 665x^5 + 1726x^4 + 2684x^3 + 2376x^2 +    1024x +    128$ & Resultant $2$ & $(x + 2)^3$                    & $(x + 4)^2 (x^3 + 5x^2 + 6x + 1)$                        \\
14. & $x^8 + 19x^7 + 152x^6 + 665x^5 + 1727x^4 + 2694x^3 + 2412x^2 +    1080x +    160$ & Resultant $1$ & $(x^2 + 5x + 5)$               & $(x + 2)^3 (x + 4) (x^2 + 4x + 1)$                       \\
15. & $x^8 + 19x^7 + 152x^6 + 665x^5 + 1727x^4 + 2695x^3 + 2420x^2 +    1100x +    176$ & Resultant $1$ & $(x + 2)^2 (x + 4)$            & $(x^5 + 11x^4 + 44x^3 + 77x^2 + 55x + 11)$               \\
16. & $x^8 + 19x^7 + 152x^6 + 665x^5 + 1727x^4 + 2696x^3 + 2428x^2 +    1120x +    192$ & Resultant $2$ & $(x + 2)^2 (x + 3)$            & $(x + 1)(x + 4)^2(x^2 + 3x + 1)$                         \\
17. & $x^8 + 19x^7 + 152x^6 + 665x^5 + 1728x^4 + 2705x^3 + 2455x^2 +    1150x +    200$ & Resultant $2$ & $(x + 2)$                      & $(x + 4)(x^2 + 3x + 1)(x^2 + 5x + 5)^2$                  \\
18. & $x^8 + 19x^7 + 152x^6 + 665x^5 + 1728x^4 + 2706x^3 + 2463x^2 +    1170x +    216$ & Resultant $2$ & $(x + 1)(x + 4)$               & $(x + 2)(x + 3)^2(x^3 + 6x^2 + 9x + 3)$                  \\
19. & $x^8 + 19x^7 + 152x^6 + 665x^5 + 1729x^4 + 2715x^3 + 2490x^2 +    1200x +    225$ & Resultant $1$ & $(x^2 + 5x + 5)^2$             & $(x + 3)(x^3 + 6x^2 + 9x + 3)$                           \\
20. & $x^8 + 19x^7 + 152x^6 + 665x^5 + 1729x^4 + 2716x^3 + 2499x^2 +    1225x +    245$ & Resultant $1$ & $(x^3 + 7x^2 + 14x + 7)^2$     & $(x^2 + 5x + 5)$                                         \\
21. & $x^8 + 19x^7 + 152x^6 + 665x^5 + 1729x^4 + 2717x^3 + 2506x^2 +    1239x +    252$ & Resultant $1$ & $(x^3 + 7x^2 + 14x + 7)$       & $(x + 1)^2 (x + 3)^2 (x + 4)$                            \\
22. & $x^8 + 19x^7 + 152x^6 + 666x^5 + 1740x^4 + 2759x^3 + 2568x^2 +    1260x +    240$ & Resultant $1$ & $(x + 2)^2 (x + 4)$            & $(x^2 + 5x + 5) (x^3 + 6x^2 + 9x + 3)$                   \\
23. & $x^8 + 19x^7 + 152x^6 + 666x^5 + 1741x^4 + 2770x^3 + 2611x^2 +    1330x +    280$ & Resultant $1$ & $(x + 1)(x + 2)(x + 4)$        & $(x^2 + 5x + 5)(x^3 + 7x^2 + 14x + 7)$                   \\
24. & $x^8 + 19x^7 + 152x^6 + 666x^5 + 1741x^4 + 2771x^3 + 2618x^2 +    1344x +    288$ & Resultant $2$ & $(x + 2)(x + 3)^2$             & $(x + 1)^3(x + 4)^2$                                     \\
25. & $x^8 + 19x^7 + 152x^6 + 666x^5 + 1742x^4 + 2780x^3 + 2645x^2 +    1375x +    300$ & Resultant $1$ & $(x^2 + 5x + 5)^2$             & $(x + 1)^2 (x + 3) (x + 4)$                              \\
26. & $x^8 + 19x^7 + 152x^6 + 667x^5 + 1753x^4 + 2824x^3 + 2724x^2 +    1440x +    320$ & Resultant $1$ & $(x^2 + 5x + 5)$               & $(x + 1)^2 (x + 2)^2 (x + 4)^2$                          \\
\bottomrule
\end{tabular}
\end{sidewaystable}

As the table indicates, all but the first of the $26$ possible real Weil 
polynomials can be eliminated from consideration by using one of the arguments
mentioned above. The entries that are eliminated by the resultant~$1$ argument
or the supersingular factor argument need no explanation beyond the 
specification of the splitting $h = h_1 h_2$ that makes the argument work. The
entries that are eliminated by the resultant~$2$ argument do require some more
explanation, to indicate why there would be a problem with a curve with real 
Weil polynomial $h$ being a double cover of a curve with real Weil polynomial
$h_1$ or~$h_2$.

The Riemann--Hurwitz formula shows that a genus-$8$ curve cannot be a double
cover of a curve of genus $5$ or larger. That fact is enough to show that for
table entries 13, 16, 17, 18, and 24, it is not possible for a curve with real
Weil polynomial $h$ to be a double cover of a curve with real Weil 
polynomial~$h_2$.

If a genus-$8$ curve with $24$ rational points is a double cover of a curve~$D$,
then the curve $D$ must have at least $12$ rational points. This shows that for
table entries 13, 17, and 18, it is not possible for a curve with real Weil 
polynomial $h$ to be a double cover of a curve with real Weil polynomial~$h_1$.

For table entries 16 and 24, there is no curve with real Weil polynomial equal 
to~$h_1$, because of the resultant $1$ argument.

The only remaining ``resultant $2$'' entry on the table is number 6. For that
entry, the supersingular factor argument shows there is no curve with real Weil
polynomial $h_2$. Next, we check that a genus-$4$ curve $D$ with real Weil
polynomial $h_1$ must have $13$ rational points. We also check that a genus-$8$
curve $C$ with real Weil polynomial $h$ has no places of degree~$2$. Therefore,
if we have a double cover $C\to D$, we see that $11$ rational points of $D$ 
split and $2$ ramify, in order to produce the $24$ rational points on~$C$. But
the ramification of the cover $C\to D$ is necessarily wild, so the 
Riemann--Hurwitz formula shows that there can be at most one point of $C$ that 
ramifies in the cover. This contradiction shows that the isogeny class for table
entry 6 does not contain a Jacobian.

Thus, the only possible real Weil polynomial for a genus-$8$ curve $C$ over 
$\BF_4$ with $24$ points is the one given for the first entry in 
Table~\ref{Table:26}, namely
\[
h = x^8 + 19x^7 + 152x^6 + 664x^5 + 1712x^4 + 2608x^3 + 2176x^2 + 768x.
\]
Note that we can write $h = h_1 h_2$, where $h_1 = x + 3$ and 
$h_2 = x(x+2)^4(x+4)^2$. Since $h_1$ is the real Weil polynomial of the unique
elliptic curve $E$ over $\BF_4$ with $8$ points, and since the resultant of 
$h_1$ and $h_2$ is $\pm 3$, Propositions~2.5 and~2.8 of~\cite{HoweLauter2012}
show that there exists a degree-$3$ map from $C$ to $E$. Let 
$\varphi\colon C\to E$ be such a triple cover. The remainder of the proof will
focus on this map~$\varphi$. First we show that $\varphi$ is not a Galois cover.

Suppose, to obtain a contradiction, that the triple cover $\varphi$ is 
Galois. Then all ramification in the cover is tame, and the Riemann--Hurwitz 
formula shows that seven geometric points of $E$ ramify.

The real Weil polynomial for $C$ shows that $C$ has $24$ places of degree~$1$ 
and no places of degree~$2$ or~$3$. Since $E$ has $8$ places of degree~$1$, all 
of them must split in the extension in order to account for the $24$ degree-$1$ 
places on~$C$, so no degree-$1$ place of $E$ ramifies. No places of $E$ of 
degree~$2$ or~$3$ ramify either, because $C$ has no places of degree~$2$ or~$3$.
Therefore, the ramification divisor on $E$ consists of a single place of 
degree~$7$. 

Every place of degree~$7$ on $C$ lies over a place of degree~$7$ on $E$, so the
number of degree-$7$ places on $C$ is equal to three times the number of 
degree-$7$ places on $E$ that split, plus the number of degree-$7$ places on $E$
that ramify. Therefore, the number of degree-$7$ places of $C$ must be $1$ 
modulo~$3$. However, the real Weil polynomial $h$ tells us that $C$ has $2496$
places of degree~$7$, and this number is $0$ modulo~$3$. This contradiction 
shows that $\varphi$ cannot be a Galois cover.

Let $k(E)$ and $k(C)$ be the function fields of $E$ and $C$, and view $k(E)$ as 
a degree-$3$ subfield of $k(C)$ via the embedding $\varphi^*$. Let $M$ be the 
Galois closure of $k(C)$ over $k(E)$, and let $L$ be the quadratic resolvent 
field. Then we have a field diagram
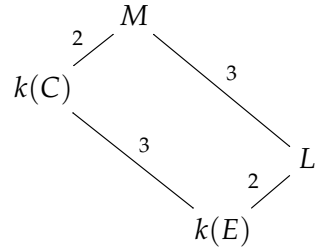
\[
\xymatrix@=1em{
                    & M\ar@{-}[ddrr]^3\ar@{-}[dl]_2 &      &                \\
k(C)\ar@{-}[ddrr]^3 &                               &      &                \\
                    &                               &      & L\ar@{-}[dl]_2 \\
                    &                               & k(E) &                \\
}
\]
where $M/k(E)$ is a Galois extension with group $S_3$.

Consider how a place $P$ of $E$ decomposes in $M$. There are six possible 
combinations of decomposition group and inertia group; these six possibilities
are listed in Table~\ref{Table:splitting}, together with the splitting behavior
in $C$ determined by the combination, and the associated contribution to the
degrees of the different $\different_C$ of $k(C)/k(E)$ and the different 
$\different_L$ of~$L/k(E)$.

\begin{table}[tb]
\centering
\caption{
   The possible decomposition and inertia groups in $M/k(E)$ of a place $P$ of
   $E$, together with associated data. Here $C_i$ denotes a cyclic group of 
   order~$i$. The columns labeled $Q_1$, $Q_2$, and $Q_3$ give the ramification
   index $e$ and residue class field degree $f$ of the (one, two, or three) 
   places of $C$ lying over~$P$, and the final two columns give the contribution
   that the places over $P$ make to the degrees of the differents of $k(C)/k(E)$
   and of $L/k(E)$. The symbol $m_P$ indicates an integer depending on $P$ that
   is computed from the higher ramification groups of $P$; we know that 
   $m_P\ge 2$ for every place $P$ with inertia group $C_2$.
}
\label{Table:splitting}
\begin{tabular}{ccccccccccccrr}
\toprule
\multicolumn{1}{c}{Decomposition}&\multicolumn{1}{c}{Inertia}&\hbox to 1em{}&\multicolumn{2}{c}{$Q_1$}&&\multicolumn{2}{c}{$Q_2$}&&\multicolumn{2}{c}{$Q_3$}&\hbox to 1em{}&\multicolumn{2}{c}{\quad Contribution to:}           \\
\multicolumn{1}{c}{group}        &\multicolumn{1}{c}{group}  &              & $e_1$ & $f_1$           && $e_2$ & $f_2$           && $e_3$ & $f_3$           &              & $\degree\different_C$ & $\degree\different_L$ \\
\midrule
$S_3$ & $C_3$ && $3$ & $1$ &&     &     &&     &     && $2\deg P$   & $0$          \\
$C_3$ & $C_3$ && $3$ & $1$ &&     &     &&     &     && $2\deg P$   & $0$          \\
$C_3$ & $C_1$ && $1$ & $3$ &&     &     &&     &     && $0$         & $0$          \\
$C_2$ & $C_2$ && $2$ & $1$ && $1$ & $1$ &&     &     && $m_P\deg P$ & $m_P \deg P$ \\
$C_2$ & $C_1$ && $1$ & $2$ && $1$ & $1$ &&     &     && $0$         & $0$          \\
$C_1$ & $C_1$ && $1$ & $1$ && $1$ & $1$ && $1$ & $1$ && $0$         & $0$          \\
\bottomrule
\end{tabular}
\end{table}

As we noted earlier, $C$ has $24$ places of degree~$1$ and no places of 
degree~$2$ or~$3$, and every place of degree~$1$ on $E$ must split completely
in~$C$. Each degree-$1$ place $P$ of $E$ must also split completely in $M$, and
in particular the places of $M$ over $P$ have residue class field degree~$1$. 
This shows that $L/k(E)$ is not a constant field extension, so $L$ and $M$ are 
function fields of geometrically irreducible curves over $\BF_4$; let $F$ and 
$D$ be these curves, so that $L = k(F)$ and $M = k(D)$.

We also see that once again, no place of $E$ of degree~$1$, $2$, or~$3$ can 
ramify in~$C$. The Riemann--Hurwitz formula says that the degree of 
$\different_C$ is $14$. Each place $P$ of $E$ that ramifies in $C$ contributes
at least $2 \deg P\ge 8$ to $\degree\different_C$, so we see that exactly one 
place must ramify. That place $P$ must have degree~$7$, and if the ramification
is wild then we must have $m_P = 2$.

Suppose the ramifying place $P$ has inertia group $C_3$. Then $L/k(E)$ is 
unramified, so $M/k(C)$ is unramified, so the curve $D$ has genus~$15$. Since
every rational point of $E$ splits completely in~$D$, we see that $D$ has $48$
rational points. But the Oesterl\'e bound for a curve of genus $15$ over $\BF_4$
is~$37$, a contradiction. Therefore the ramifying place $P$ must have inertia 
group~$C_2$.

Since the only ramifying place has inertia group~$C_2$, 
Table~\ref{Table:splitting} shows that 
$\degree\different_L = \degree\different_C = 14$, and it follows from the 
Riemann--Hurwitz formula that the genus of the curve $F$ is~$8$.

Since $C$ has no places of degree $2$ or~$3$, Table~\ref{Table:splitting} shows
that the places of degree $2$ and~$3$ on $E$ must have decomposition group $C_3$
in $M/k(E)$, so they split in $L/k(E)$. Earlier we showed that the degree-$1$
places of $E$ split completely in~$M$, and hence also in~$L$. Since $E$ has $8$
degree-$1$ places, $4$ degree-$2$ places, and $16$ degree-$3$ places, it follows
that $F$ has $16$ degree-$1$ places, $8$ degree-$2$ places, and $32$ degree-$3$
places.

The Magma code that accompanies~\cite{HoweLauter2012} includes an implementation
of Lauter's algorithm~\cite{Lauter2000} for enumerating all Weil polynomials $f$
of $g$-dimensional isogeny classes over $\BF_q$ with $P_n(f)\ge 0$ for all $n$ 
and with $P_1(f)$ equal to a given integer $N$. This can be easily adapted to 
enumerate all Weil polynomials $f$ of $8$-dimensional isogeny classes over 
$\BF_4$ with $P_n(f)\ge 0$ for all $n$ and with $P_1(f) = 16$, $P_2(f) = 8$,
and $P_3(f) = 32$. We find that there are $44$ such Weil polynomials; the 
associated real Weil polynomials are listed in Table~\ref{Table:realWPs}.

\begin{table}[tb]
\centering
\caption{
   Possible real Weil polynomials for a genus-$8$ curve over $\BF_4$ that has
   $16$ degree-$1$ places, $8$ degree-$2$ places, and $32$ degree-$3$ places.
}
\label{Table:realWPs}
\scriptsize
\begin{tabular}{l@{\qquad\qquad}l}
\toprule
$x^8 + 11x^7 + 36x^6 + 12x^5 - 96x^4 - 64x^3                               $ & $x^8 + 11x^7 + 36x^6 + 12x^5 - 93x^4 - 40x^3 +     52x^2 +     15x - {\p}4 $ \\
$x^8 + 11x^7 + 36x^6 + 12x^5 - 96x^4 - 64x^3 + {\p\p}x^2 +  {\p}4x         $ & $x^8 + 11x^7 + 36x^6 + 12x^5 - 93x^4 - 40x^3 +     52x^2 +     16x         $ \\
$x^8 + 11x^7 + 36x^6 + 12x^5 - 95x^4 - 56x^3 +     16x^2                   $ & $x^8 + 11x^7 + 36x^6 + 12x^5 - 93x^4 - 39x^3 +     56x^2 +     15x - {\p}4 $ \\
$x^8 + 11x^7 + 36x^6 + 12x^5 - 95x^4 - 56x^3 +     17x^2 +  {\p}4x         $ & $x^8 + 11x^7 + 36x^6 + 12x^5 - 93x^4 - 39x^3 +     56x^2 +     16x         $ \\
$x^8 + 11x^7 + 36x^6 + 12x^5 - 95x^4 - 56x^3 +     19x^2 +     12x         $ & $x^8 + 11x^7 + 36x^6 + 12x^5 - 93x^4 - 39x^3 +     57x^2 +     19x - {\p}4 $ \\
$x^8 + 11x^7 + 36x^6 + 12x^5 - 94x^4 - 49x^3 +     28x^2                   $ & $x^8 + 11x^7 + 36x^6 + 12x^5 - 92x^4 - 33x^3 +     57x^2 -     12x         $ \\
$x^8 + 11x^7 + 36x^6 + 12x^5 - 94x^4 - 48x^3 +     33x^2 +  {\p}4x         $ & $x^8 + 11x^7 + 36x^6 + 12x^5 - 92x^4 - 33x^3 +     58x^2 -  {\p}8x         $ \\
$x^8 + 11x^7 + 36x^6 + 12x^5 - 94x^4 - 48x^3 +     35x^2 +     11x - {\p}4 $ & $x^8 + 11x^7 + 36x^6 + 12x^5 - 92x^4 - 33x^3 +     59x^2 -  {\p}4x         $ \\
$x^8 + 11x^7 + 36x^6 + 12x^5 - 94x^4 - 48x^3 +     35x^2 +     12x         $ & $x^8 + 11x^7 + 36x^6 + 12x^5 - 92x^4 - 33x^3 +     60x^2 - {\p\p}x - {\p}4 $ \\
$x^8 + 11x^7 + 36x^6 + 12x^5 - 94x^4 - 47x^3 +     40x^2 +     15x - {\p}4 $ & $x^8 + 11x^7 + 36x^6 + 12x^5 - 92x^4 - 33x^3 +     60x^2                   $ \\
$x^8 + 11x^7 + 36x^6 + 12x^5 - 94x^4 - 47x^3 +     40x^2 +     16x         $ & $x^8 + 11x^7 + 36x^6 + 12x^5 - 92x^4 - 33x^3 +     61x^2 +  {\p}4x         $ \\
$x^8 + 11x^7 + 36x^6 + 12x^5 - 94x^4 - 47x^3 +     41x^2 +     20x         $ & $x^8 + 11x^7 + 36x^6 + 12x^5 - 92x^4 - 32x^3 +     64x^2 - {\p\p}x - {\p}4 $ \\
$x^8 + 11x^7 + 36x^6 + 12x^5 - 94x^4 - 47x^3 +     43x^2 +     29x + {\p}4 $ & $x^8 + 11x^7 + 36x^6 + 12x^5 - 92x^4 - 32x^3 +     64x^2                   $ \\
$x^8 + 11x^7 + 36x^6 + 12x^5 - 93x^4 - 41x^3 +     43x^2 -  {\p}4x         $ & $x^8 + 11x^7 + 36x^6 + 12x^5 - 92x^4 - 32x^3 +     65x^2 +  {\p}2x - {\p}8 $ \\
$x^8 + 11x^7 + 36x^6 + 12x^5 - 93x^4 - 41x^3 +     44x^2                   $ & $x^8 + 11x^7 + 36x^6 + 12x^5 - 92x^4 - 32x^3 +     65x^2 +  {\p}3x - {\p}4 $ \\
$x^8 + 11x^7 + 36x^6 + 12x^5 - 93x^4 - 40x^3 +     48x^2                   $ & $x^8 + 11x^7 + 36x^6 + 12x^5 - 92x^4 - 32x^3 +     67x^2 +  {\p}9x -    12 $ \\
$x^8 + 11x^7 + 36x^6 + 12x^5 - 93x^4 - 40x^3 +     49x^2 +  {\p}3x - {\p}4 $ & $x^8 + 11x^7 + 36x^6 + 12x^5 - 92x^4 - 32x^3 +     68x^2 +     12x -    16 $ \\
$x^8 + 11x^7 + 36x^6 + 12x^5 - 93x^4 - 40x^3 +     49x^2 +  {\p}4x         $ & $x^8 + 11x^7 + 36x^6 + 12x^5 - 91x^4 - 26x^3 +     68x^2 -     16x         $ \\
$x^8 + 11x^7 + 36x^6 + 12x^5 - 93x^4 - 40x^3 +     50x^2 +  {\p}7x - {\p}4 $ & $x^8 + 11x^7 + 36x^6 + 12x^5 - 91x^4 - 26x^3 +     69x^2 -     12x         $ \\
$x^8 + 11x^7 + 36x^6 + 12x^5 - 93x^4 - 40x^3 +     51x^2 +     10x - {\p}8 $ & $x^8 + 11x^7 + 36x^6 + 12x^5 - 91x^4 - 25x^3 +     72x^2 -     16x         $ \\
$x^8 + 11x^7 + 36x^6 + 12x^5 - 93x^4 - 40x^3 +     51x^2 +     11x - {\p}4 $ & $x^8 + 11x^7 + 36x^6 + 12x^5 - 91x^4 - 25x^3 +     73x^2 -     13x - {\p}4 $ \\
$x^8 + 11x^7 + 36x^6 + 12x^5 - 93x^4 - 40x^3 +     51x^2 +     12x         $ & $x^8 + 11x^7 + 36x^6 + 12x^5 - 90x^4 - 19x^3 +     78x^2 -     24x         $ \\
\bottomrule
\end{tabular}
\end{table}

Our genus-$8$ curve $F$ is a double cover of the elliptic curve $E$, and 
therefore the real Weil polynomial for $E$ --- namely, $x+3$ --- should divide
that of~$F$. But we check that none of the $44$ real Weil polynomials listed in
Table~\ref{Table:realWPs} is divisible by $x + 3$. This contradiction shows that
the map $\varphi\colon C\to E$ cannot exist, so our hypothetical genus-$8$ curve
over $\BF_4$ with $24$ points cannot exist either.
\end{proof}

We note that the Magma routines accompanying~\cite{HoweLauter2012} automatically
produce an argument, similar to the first half of the proof of 
Theorem~\ref{T:upper}, that shows that a genus-$8$ curve over $\BF_4$ with $24$
points is a triple cover of an elliptic curve with $8$ points.

\section{Improving the lower bound}
\label{S:example}

In this section we provide an example that proves the lower bound in 
Theorem~\ref{T:intro}.

\begin{theorem}
\label{T:example}
The curve over $\BF_4$ defined by the two equations 
\[
y^2 + (x^3 + x + 1)y = x^6 + x^5 + x^4 + x^2
\text{\qquad and\qquad}
z^3 = (x+1)y + x^2
\]
has genus~$8$ and has $22$ rational points.
\end{theorem}

\begin{proof}
Let $D$ be the genus-$2$ curve $y^2 + (x^3 + x + 1)y = x^6 + x^5 + x^4 + x^2$
and let $C$ be the curve in the theorem, so that $C$ is the degree-$3$ Kummer
extension of $D$ obtained by adjoining a cube root of the function 
$f = (x+1)y + x^2$. We check that 
\[
\divisor f = 6P_0 + P_1 + P_2 - 4Q_1 - 4Q_2,
\]
where $Q_1$ and $Q_2$ are the two (rational) points at infinity on $D$, and 
where $P_0$, $P_1$, and $P_2$ are the points $(0,0)$, $(\omega,1)$, and 
$(\omega^2,1)$, where $\omega$ and $\omega^2$ are the elements of $\BF_4$ not
in~$\BF_2$. Thus, the points $P_1$, $P_2$, $Q_1$, and $Q_2$ ramify totally and
tamely in the cover $C/D$, and the Riemann--Hurwitz formula shows that $C$ has
genus~$8$.

The function $f$ evaluates to $1$ on the five rational points of $D$ that are
not in the support of $f$, so each of those five points splits completely 
in~$C$. The point $P_0$ also splits in $C$; the function $x$ is a uniformizer 
at~$P_0$, and we have $y = x^2 + x^3 + x^4 + x^5 + O(x^7)$, so 
$f = x^6 + O(x^7)$ and $f$ is locally a cube at $P_0$. Thus, six rational points
of $D$ split completely, and four rational points of $D$ ramify, leading to 
$6\cdot 3 + 4 = 22$ points on $C$.
\end{proof}

\bibliography{genus8F4}
\bibliographystyle{hplaindoi}
\end{document}